\documentclass{article}
\usepackage{amsmath}
\usepackage{amsfonts}
\usepackage{amssymb}
\usepackage{mathrsfs}
\usepackage{amsthm}
\usepackage{latexsym, amsfonts, amsmath, amssymb, amscd, epsfig, subfig}
\usepackage{graphicx}
\usepackage{yhmath}

\newtheorem{theorem}{Theorem}[section]
\newtheorem{definition}[theorem]{Definition}
\newtheorem{lemma}[theorem]{Lemma}

\newtheorem{remark}[theorem]{Remark}
\numberwithin{equation}{section}

\allowdisplaybreaks %%%%%%%%%%% eqnarray公式 强制换页%%%%%%%%%%%

\title{\bf Nonsymmetric extension of the Green-Osher inequality
\footnote{The author is supported by the Doctoral Scientific Research Foundation of
Liaoning Province (No.20170520382)
and the Fundamental Research
Funds for the Central Universities (No.3132017046).}}

\author{
\bf Yunlong Yang\\
College of Science,\ \ Dalian Maritime University,\\
Dalian, 116026, P. R. China \\email: ylyang@dlmu.edu.cn\\
}
\date{}

\begin{document}
\maketitle \noindent

{\noindent \bf Abstract} \quad In this paper we
obtain the extended Green-Osher inequality
when two smooth, planar strictly convex bodies are at a
dilation position and show the necessary and sufficient
condition for the case of equality.
\\\\
{\bf Mathematics\ Subject\ Classification 2010:}\ 52A40, \ 52A10

\section{Introduction}\label{sec1}

We denote by $\mathbb{R}^n$ the usual $n$-dimensional Euclidean space with
the canonical inner product $\langle\cdot,\cdot\rangle$. A compact convex
set $K$ in $\mathbb{R}^n$ is called a {\it convex body} if it contains the
origin and has nonempty interior. When $n = 2$, it is called a {\it planar convex body}.
The volume of a set $S \subseteq\mathbb{R}^n$ is denoted by $V(S)$.
The {\it Minkowski sum} of convex bodies $K$ and $L$, and
the {\it Minkowski scalar product} of $K$ for positive real number $t$ are,
respectively, defined by
\begin{equation*}
    K+L=\{x+y\mid x\in K,y\in L\}
\end{equation*}
and
\begin{equation*}
    t K=\{t x\mid x\in K\}.
\end{equation*}
For two planar convex bodies $K$ and $L$, the volume of the Minkowski sum $K + tL$
gives the {\it relative Steiner polynomial} of $K$ with respect to $L$:
\begin{equation}\label{eqn1.1}
    V(K+tL)=V(K)+2V(K,L)t+V(L)t^2,
\end{equation}
where $V(K,L)$ is the {\it mixed area} of $K$ and $L$. Formula \eqref{eqn1.1} is
closely related to the classical isoperimetric inequality, the Brunn-Minkowski inequality
and the log-Brunn-Minkowski inequality.
Many proofs, sharpened forms and generalization
of the isoperimetric inequality can be found in Chavel \cite{C2001},
Dergiades \cite{D2002}, Osserman \cite{O1978}
and Schneider \cite{S2014}.

Using remarkable symmetrization, Gage \cite{Gage1983} successfully obtained an inequality for the
total squared curvature for convex curves. Following his work, for a planar strictly convex body $K$
and a symmetric, planar strictly convex body $E$, Green and Osher \cite{G-O1999} (see also \cite{Y-Z2016})
obtained a generalized formula:
\begin{equation}\label{eqn1.2}
\frac{1}{ V(E)}\int_0^{2\pi} F(\rho(\theta))h_E(\theta)(h_E(\theta)+h_E''(\theta))\mathrm{d}\theta \ge F(-t_1)+F(-t_2),
\end{equation}
where $\rho(\theta)$ is the relative curvature radius of $K$ with respect to $E$,
$F(x)$ is a strictly convex function on $(0,+\infty)$,
$t_1$ and $t_2$ are the two roots of the relative Steiner polynomial of $K$ with
respect to $E$. Inequality \eqref{eqn1.2} plays a significant role in studying the curve
shortening flow (see Gage \cite{Gage1984,G1993} and Gage-Hamilton \cite{G-H1986}).

A natural question is whether the Green-Osher inequality holds
without symmetric condition. Similar question is asked by the
log-Brunn-Minkowski inequality (see B\"{o}r\"{o}czky-Lutwak-Yang-Zhang \cite{B-L-Y-Z2012},
Xi-Leng \cite{X-L2016} and Yang-Zhang \cite{Y-Z2018}).
Xi and Leng \cite{X-L2016} gave the definition of dilation position for the first time
to prove the log-Brunn-Minkowski inequality and solve the planar Dar's conjecture.

Let $K$ and $L$ be two convex bodies.
Convex bodies $K$ and $L$ are at a {\it dilation position},
if the origin $o\in K\cap L$ and
\begin{equation}\label{eqn7.0}
    r(K,L)L\subseteq K\subseteq R(K,L)L.
\end{equation}
Here $r(K,L)$ and $R(K,L)$ are the {\it inradius}
and {\it outradius} of $K$ with respect to $L$, i.e.,
\begin{align*}
    &r(K,L)=\max\{t>0\mid x+tL\subseteq K~\text{and}~ x\in \mathbb{R}^n\},\\
    &R(K,L)=\max\{t>0\mid x+tL\supseteq K~ \text{and}~ x\in \mathbb{R}^n\}
\end{align*}
Noticing that there is a common center
when $K$ and $L$ are at a dilation position, then
the ratio of the support functions of $K$ and $L$
belongs to the range from $r(K,L)$ to $R(K,L)$,
which leads to the Green-Osher inequality holds
without symmetric condition. Properties of convex
bodies are at a dilation position can be found
in Lemma \ref{lem7.2} (see also Xi-Leng \cite{X-L2016}).

In this paper, inspired by the impressive work in \cite{X-L2016},
we obtain the main result.

\begin{theorem}\label{thm7.5}
Let $K, L$ be two smooth, planar strictly convex bodies and
$\rho(\theta)$ the relative curvature radius of $K$ with respect to $L$.
If $K$ and $L$ are at a dilation position and $F(x)$ is a strictly convex
function on $(0,+\infty)$, then
\begin{equation}\label{eqn7.5}
    \frac{1}{V(L)}\int_{0}^{2\pi} F(\rho(\theta))h_L(\theta)
    (h_L(\theta)+h_L''(\theta))\mathrm{d}\theta\ge
    F(-t_1)+F(-t_2),
\end{equation}
where $t_1$ and $t_2$ are the two roots of the relative
Steiner polynomial of $K$ with respect to $L$, and
the equality in \eqref{eqn7.5} holds if and only
if $K$ and $L$ are homothetic.
\end{theorem}

This paper is organized as follows. In Section \ref{sec2},
we give some basic facts about planar convex bodies.
In Section \ref{sec3}, we get the extended Green-Osher inequality
when two smooth, planar strictly convex bodies are at a
dilation position.

\section{Preliminaries}\label{sec2}

Let $K$ be a planar convex body.
A line $l$ is called a {\it support line} of $K$ if it passes through
at least one boundary point of $K$ and if the entire planar convex body
$K$ lies on one side of $l$.
Let $l(\theta)$ be the support line of $K$
in the direction $\mathbf{u}(\theta)=(\cos\theta, \sin\theta)$,
where $\theta$ is the oriented angle from the positive
$x$-axis to the perpendicular line of $l(\theta)$.
The {\it support function} of $K$ is defined by
$$h_K(\theta)= \sup_{x\in K} \langle x, \mathbf{u}(\theta)\rangle, \quad \mathbf{u}(\theta)\in S^1.$$
It is easy to see that $h_K(\theta)$ is the signed distance
of the support line $l(\theta)$ of $K$ with
exterior normal vector $\mathbf{u}(\theta)$ from
the origin. Clearly, $h_K$, as a function of $\theta$,
is single-valued and $2\pi$-periodic.

If $h_K(\theta)$ and $h_L(\theta)$
are continuously differentiable, then
\begin{equation*}
    V(K,L)=\frac{1}{2}\int_{0}^{2\pi}
    (h_K(\theta)h_L(\theta)-h_K'(\theta)h_L'(\theta))
    \mathrm{d}\theta.
\end{equation*}
Furthermore, if $h_K(\theta)$ and $h_L(\theta)$ are smooth,
then
\begin{equation*}
    V(K,L)=\frac{1}{2}\int_{0}^{2\pi}
    h_K(\theta)(h_L(\theta)+h_L''(\theta))
    \mathrm{d}\theta=\frac{1}{2}\int_{0}^{2\pi}
    h_L(\theta)(h_K(\theta)+h_K''(\theta))
    \mathrm{d}\theta.
\end{equation*}

From the Minkowski inequality, 
it follows that
the expression $V(K+tL) = 0$ has two negative real roots.
Denote by $t_1$ and $t_2$ ($t_1\ge t_2$) the two roots of
the relative Steiner polynomial of $K$ with respect
to $L$, that is,
\begin{equation*}
    t_1= -\frac{V(K,L)}{V(L)}+\frac{\delta}{V(L)},\quad
    t_2= -\frac{V(K,L)}{V(L)}-\frac{\delta}{V(L)},\quad
    \delta=\sqrt{V(K,L)^2-V(K)V(L)}.
\end{equation*}

In order to prove the extended Green-Osher inequality,
we have the following definition that is
similar to the Definition 3.3 of \cite{G-O1999}.

\begin{definition}[\cite{G-O1999}]\label{def7.1}
Let $K,\,L$ be two smooth, planar strictly convex bodies. Consider
\begin{equation*}
    \sup \left\{\int_I \rho(\theta)h_L(\theta)(h_L(\theta)+h_L''(\theta))
    \mathrm{d}\theta \mid I\subseteq S^1,
    \int_I h_L(\theta)(h_L(\theta)+h_L''(\theta))
    \mathrm{d}\theta=V(L)\right\}.
\end{equation*}
Let $I_1$ denote the smallest subset of $S^1$ with measure $V(L)$ and realizing
the above supremum, and let $I_2$ be its complement. Then, there exists an $a\in \mathbb{R}^+$
such that
\begin{equation*}
    I_1\subseteq \{\theta\mid \rho(\theta)\ge a\},\quad
    I_2\subseteq \{\theta\mid \rho(\theta)\le a\}.
\end{equation*}
\end{definition}
Set
\begin{equation*}
    \rho_i=\frac{1}{V(L)}\int_{I_i} \rho(\theta)h_L(\theta)(h_L(\theta)+h_L''(\theta)) \mathrm{d}\theta,\quad i=1,2,
\end{equation*}
which yield that
\begin{equation*}
    \rho_1+\rho_2=\frac{2V(K,L)}{V(L)} \quad\text{and}\quad \rho_1\ge \rho_2,
\end{equation*}
and there is a real number $b\ge 0$ such that
\begin{equation*}
    \rho_1= \frac{V(K,L)}{V(L)}+ b \quad\text{and}\quad  \rho_2=\frac{V(K,L)}{V(L)}-b.
\end{equation*}

\section{Nonsymmetric extension of the Green-Osher inequality}\label{sec3}

In order to prove the main result, we first give four lemmas, in which
Lemma \ref{lem7.2} shows that convex bodies are at a dilation position
by appropriate translations and the location of ``dilation position'' (detailed proof can be
found in \cite[Lemma 2.1]{X-L2016}), Lemmas \ref{lem7.3} and \ref{lem7.4} are
used to prove inequality \eqref{eqn7.5},
and Lemma \ref{lem7.5} is used to deal with its equality case.

\begin{lemma}[\cite{X-L2016}]\label{lem7.2}
Let $K,L$ be two convex bodies in $\mathbb{R}^n$. %the $n$-dimensional Euclidean space.
\begin{itemize}
  \item [(i)] There are a translate of $L$, say $\bar{L}$,
  and a translate of $K$, say $\bar{K}$, so
  that $\bar{K}$ and $\bar{L}$ are at a dilation position.
  \item [(ii)] If $K$ and $L$ are at a dilation position,
  then the origin $o\in \mathrm{int}(K\cap L)\cup(\partial K\cap \partial L)$.
\end{itemize}
\end{lemma}

\begin{lemma}\label{lem7.3}
Let $K, L$ be two smooth, planar strictly convex bodies.
If $K$ and $L$ are at a dilation position, then
the origin $o\in \mathrm{int}(K\cap L)$ or $o$ is the point
of tangency of $\partial K$ and $\partial L$ such that $K\subseteq L$
(or $L\subseteq K$).
\end{lemma}

\begin{proof}
By Lemma \ref{lem7.2}(ii), the origin
$o\in \mathrm{int}(K\cap L)\cup(\partial K\cap \partial L)$.
If the origin $o\in \mathrm{int}(K\cap L)$,
we are done. If the origin $o\in\partial K\cap \partial L$,
then $o$ must be the point of tangency of $\partial K$
and $\partial L$ such that $K\subseteq L$
(or $L\subseteq K$). Otherwise, $o$ is the point
of intersection of $\partial K$ and $\partial L$,
which contradicts to \eqref{eqn7.0}.
\end{proof}

\begin{lemma}\label{lem7.4}
Let $K, L$ be two smooth, planar strictly convex bodies.
If $K$ and $L$ are at a dilation position, then
\begin{equation}\label{eqn7.1}
    \rho_1\ge -t_2.
\end{equation}
\end{lemma}

\begin{proof}
From \cite[Lemma 4.1]{B-L-Y-Z2012}
and the Minkowski inequality, it follows that
$$-t_1\le r(K,L)\le R(K,L)\le -t_2.$$
By Lemma \ref{lem7.3}, the origin $o\in \mathrm{int}(K\cap L)$ or $o$ is the point
of tangency of $\partial K$ and $\partial L$ such that $K\subseteq L$
(or $L\subseteq K$).

If the origin $o\in \mathrm{int}(K\cap L)$, then
$r(K,L)\le\frac{ h_K(\theta)}{h_L(\theta)}\le R(K,L)$,
which implies
\begin{equation*}
-\frac{\delta}{ V(L)}h_L(\theta)\le h_K(\theta)-\frac{V(K,L)}{V(L)}h_L(\theta) \le\frac{\delta}{ V(L)}h_L(\theta),\quad \delta=\sqrt{V(K,L)^2-V(K)V(L)}\ge 0.
\end{equation*}
On $I_1$, $\rho(\theta)-a\ge 0$, combining with the above inequality, it yields
\begin{equation*}
-\left( h_K(\theta)-\frac{V(K,L)}{V(L)}h_L(\theta)\right)(\rho(\theta)-a)\le \frac{\delta}{ V(L)}h_L(\theta)(\rho(\theta)-a).
\end{equation*}
By integrating this on the interval $I_1$, %it shows that
\begin{equation}\label{eqn7.3}
-\frac{1}{V(L)}\int_{I_1}\left( h_K(\theta)-\frac{V(K,L)}{V(L)}h_L(\theta)\right)(\rho(\theta)-a)(h_L(\theta)+h_L''(\theta))\mathrm{d}\theta
\le \frac{\delta}{V(L)}(\rho_1-a).
\end{equation}
Similarly, on $I_2$, we have
\begin{equation}\label{eqn7.4}
-\frac{1}{V(L)}\int_{I_2}\left( h_K(\theta)-\frac{V(K,L)}{V(L)}h_L(\theta)\right)(\rho(\theta)-a)(h_L(\theta)+h_L''(\theta))\mathrm{d}\theta
\le -\frac{\delta}{V(L)}(\rho_2-a).
\end{equation}
It can be seen from \eqref{eqn7.3} and \eqref{eqn7.4} that
\begin{equation*}
-\frac{1}{V(L)}\int_0^{2\pi}\left( h_K(\theta)-\frac{V(K,L)}{V(L)}h_L(\theta)\right)(\rho(\theta)-a)(h_L(\theta)+h_L''(\theta))\mathrm{d}\theta
\le \frac{2b \delta}{V(L)}
\end{equation*}
and its left-hand side can be simplified to $\frac{2\delta^2}{V(L)^2}$,
thus we have, $b\ge \frac{\delta}{V(L)}\ge 0$, that is, $\rho_1\ge -t_2$.

If the origin $o$ is the point of tangency of
$\partial K$ and $\partial L$ such that $L\subseteq K$
(the case of $K\subseteq L$ is similar),
then $r(K,L)\le\frac{ h_K(\theta)}{h_L(\theta)}\le R(K,L)$
for $\theta\in \tilde{I}$ ($\tilde{I}$ is a subset of $S^1$)
and $h_K(\theta)=h_L(\theta)=0$ for $\theta\in S^1\setminus \tilde{I}$.
A similar discussion implies that $\rho_1\ge -t_2$.
\end{proof}

\begin{lemma}\label{lem7.5}
Let $K, L$ be two smooth, planar strictly convex bodies.
If $K$ and $L$ are at a dilation position but not homothetic, then
\begin{equation}\label{eqn7.2}
    \rho_1> -t_2.
\end{equation}
\end{lemma}

\begin{proof}
Since $K$ and $L$ are not homothetic, by \cite[Lemma 4.1]{B-L-Y-Z2012}
and the fact that $K$ and $L$ are smooth and strictly convex,
$$-t_1< r(K,L)< R(K,L)< -t_2.$$
By Lemma \ref{lem7.3}, the origin $o\in \mathrm{int}(K\cap L)$ or $o$ is the point
of tangency of $\partial K$ and $\partial L$ such that $K\subseteq L$
(or $L\subseteq K$).

If the origin $o\in \mathrm{int}(K\cap L)$, then
\begin{equation*}
-\frac{\delta}{ V(L)}h_L(\theta)< h_K(\theta)-\frac{V(K,L)}{V(L)}h_L(\theta) <\frac{\delta}{ V(L)}h_L(\theta),
\quad \delta=\sqrt{V(K,L)^2-V(K)V(L)}> 0.
\end{equation*}
For $I_1$ and $I_2$,
$\rho(\theta)\equiv a$ holds on at most one interval, unless $K$ and $L$ are homothetic.
Without loss of generality, assume that $\rho(\theta)> a$ on a subinterval $I_1'$
of $I_1$. On $I_1'$, $\rho(\theta)> a$ and
\begin{equation*}
-\left( h_K(\theta)-\frac{V(K,L)}{V(L)}h_L(\theta)\right)(\rho(\theta)-a)< \frac{\delta}{ V(L)}h_L(\theta)(\rho(\theta)-a).
\end{equation*}
Integrating this expression over the interval $I_1$ yields
\begin{equation*}%\label{eqn7.3}
-\frac{1}{V(L)}\int_{I_1}\left( h_K(\theta)-\frac{V(K,L)}{V(L)}h_L(\theta)\right)(\rho(\theta)-a)(h_L(\theta)+h_L''(\theta))\mathrm{d}\theta
< \frac{\delta}{V(L)}(\rho_1-a),
\end{equation*}
which, together with \eqref{eqn7.4}, gives
\begin{equation*}
-\frac{1}{V(L)}\int_0^{2\pi}\left( h_K(\theta)-\frac{V(K,L)}{V(L)}h_L(\theta)\right)(\rho(\theta)-a)(h_L(\theta)+h_L''(\theta))\mathrm{d}\theta
<\frac{2b \delta}{V(L)}.
\end{equation*}
By a similar argument as in Lemma \ref{lem7.4}, $b > \frac{\delta}{V(L)} > 0$, which implies that $\rho_1>-t_2$.

If the origin $o$ is the point of tangency of
$\partial K$ and $\partial L$ such that $L\subseteq K$
(the case of $K\subseteq L$ is similar),
then
\begin{equation*}
-\frac{\delta}{ V(L)}h_L(\theta)< h_K(\theta)-\frac{V(K,L)}{V(L)}h_L(\theta) <\frac{\delta}{ V(L)}h_L(\theta)
\end{equation*}
for $\theta\in \tilde{I}$.
Similar with the case that the origin $o\in \mathrm{int}(K\cap L)$,
one can get $\rho_1> -t_2$.
\end{proof}

Now, we give the proof of Theorem \ref{thm7.5}.

{\noindent \bf Proof of Theorem \ref{thm7.5}}
%\begin{proof}
By Jensen's inequality on $I_i$, $i=1,2$, one has
\begin{equation*}
\frac{1}{V(L)}\int_{I_i} F(\rho(\theta))h_L(\theta)(h_L(\theta)+h_L''(\theta))\mathrm{d}\theta\ge F(\rho_i).
\end{equation*}
Then
\begin{equation}\label{eqn7.6}
\frac{1}{V(L)}\int_0^{2\pi} F(\rho(\theta))h_L(\theta)(h_L(\theta)+h_L''(\theta))\mathrm{d}\theta\ge F(\rho_1)+F(\rho_2),
\end{equation}
where $\rho_1=\frac{V(K,L)}{V(L)}+b$, $\rho_2=\frac{V(K,L)}{V(L)}-b$ and $b\ge 0$. Again from \eqref{eqn7.1},
it follows that $b\ge \frac{\delta}{V(L)}\ge 0$ and $\delta=\sqrt{V(K,L)^2-V(K)V(L)}$.
Since function $F(x)$ is strict convexity,% (see \cite[Lemma 2.9]{G-O1999}),
\begin{align}
F(\rho_1)+F(\rho_2)&=F\left(\frac{V(K,L)}{V(L)}+b\right)+F\left(\frac{V(K,L)}{V(L)}-b\right)\nonumber\\
&\ge F\left(\frac{V(K,L)}{ V(L)}+\frac{\delta}{V(L)}\right)
+F\left(\frac{V(K,L)}{V(L)}-\frac{\delta}{ V(L)}\right)
=F(-t_1)+F(-t_2),\label{eqn7.7}
\end{align}
which together with \eqref{eqn7.6} yields inequality \eqref{eqn7.5}.

On one hand, if $K$ and $L$ are homothetic, then $-t_1=-t_2=\rho(\theta)$,
it is clear that the equality holds in \eqref{eqn7.5}. On the other hand,
in order to prove that $K$ and $L$ are homothetic
when the equality holds in \eqref{eqn7.5}, it is enough to show that
inequality \eqref{eqn7.5} is strict
when $K$ and $L$ are not homothetic.
If $K$ and $L$ are not homothetic, then $\delta=\sqrt{V(K,L)^2-V(K) V(L)}>0$,
and by \eqref{eqn7.2}, one has $b>\frac{\delta}{V(L)}>0$. It follows from
the strict convexity of function $F(x)$ that \eqref{eqn7.7} is strict,
which together with \eqref{eqn7.6} implies that \eqref{eqn7.5} is strict.\qed

\begin{remark}
If $\mathbb{R}^2$ is equipped with a suitable Minkowski metric such that
$\partial L$ becomes the isoperimetrix of the Minkowski plane, then
\eqref{eqn7.5} turns into an inequality in Minkowski geometry
(see \cite[Remark 3.6]{Y-Z2016}).
\end{remark}

\section*{Acknowledgements}
I am grateful to the anonymous referee for his or her careful
reading of the original manuscript of this paper and giving us many
invaluable comments. I would also like to thank
Professor Shengliang Pan for posing this problem to me.

\end{document}